\documentclass[]{imsart}
\RequirePackage[OT1]{fontenc}
\RequirePackage{amsthm,amsmath}
\RequirePackage[numbers]{natbib}
\RequirePackage[colorlinks,citecolor=blue,urlcolor=blue]{hyperref}
\usepackage{amsmath}
\usepackage{amsthm,amscd,amssymb}
\arxiv{}
\usepackage{enumerate}
\usepackage{natbib}
\usepackage{url} 
\usepackage{mathrsfs}
\usepackage{color}
\usepackage{graphicx} 
\usepackage{bmpsize}

\newtheorem{defi}{Definition}
\newtheorem{lemma}[defi]{Lemma}
\newtheorem{remark}[defi]{Remark}

\newtheorem{theorem}[defi]{Theorem}

\newtheorem{example}[defi]{Example}

\newcommand{\E}{{\mathbb{E}}}
\newcommand{\1}{\mbox{$\mathbb{I}$}}

\newcommand{\DF}{\mathcal{F}}
\newcommand{\DG}{\mathcal{G}}

\startlocaldefs
\numberwithin{equation}{section}
\theoremstyle{plain}

\endlocaldefs

\begin{document}

\begin{frontmatter}
\title{Finite sample bounds for expected number of false rejections under martingale dependence with applications to FDR.\thanksref{t1} }
\runtitle{Finite sample bounds for ENFR with applications to FDR}

\begin{aug}
\author{\snm{Benditkis Julia}\ead[label=e1]{Benditkis@math.uni-duesseldorf}}
\and
\author{\snm{Janssen Arnold}\ead[label=e2]{Janssena@uni-duesseldorf.de}}
\address{Heinrich-Heine University D\"usseldorf,
Universit\"atsstr. 1, 40225 D\"usseldorf, Germany.\\
\printead{e1,e2}}

\thankstext{t1}{Partially suppoted by SAW-Project "Multiplizit\"at, Modellvalidierung und Reproduzierbarkeit in hochdimensionalen Microarray-Daten"}
\runauthor{J. Benditkis and A. Janssen}

\affiliation{Some University and Another University}

\end{aug}

\begin{abstract}
Much effort has been made to improve the famous step up test of Benjamini and Hochberg given by linear critical values $\frac{i\alpha}{n}$.  It is pointed out by Gavrilov, Benjamini and Sarkar that step down multiple tests based on the critical values $\beta_i=\frac{i\alpha}{n+1-i(1-\alpha)}$ still control the false discovery rate (FDR) at the upper bound $\alpha$ under basic independence assumptions. Since that result in not longer true for step up tests or dependent single tests, a big discussion about the corresponding FDR starts in the literature. The present paper establishes finite sample formulas and bounds for the FDR and the expected number of false rejections for multiple tests using critical values $\beta_i$ under martingale and reverse martingale dependence models. It is pointed out that martingale methods are natural tools for the treatment of local FDR estimators which are closely connected to the present coefficients $\beta_i.$ The martingale approach also yields new results and further inside for the special basic independence model.  
\end{abstract}


\tableofcontents
\end{frontmatter}

\section{Introduction}
Multiple tests are nowadays well established procedures for judging high dimensional data. The famous Benjamini and Hochberg \cite{b_h_95} step up multiple test given by linear critical values controls the false discovery rate FDR for various dependence models. The FDR is the expectation of the ratio of the number of false rejections devided by the amount of all rejected hypotheses. For these reasons the linear step up test is frequently applied in practice. Gavrilov et al. \cite{gavrilov_2009} pointed out that linear critical values can be substituted by
\begin{align}\label{betas}
\beta_i=\frac{i\alpha}{n+1-i(1-\alpha)}, \ i\leqslant n,
\end{align}
for step down tests and the FDR control (i.e. FDR$\leqslant \alpha$) remains true for the basic independence model of the underlying p-values. Note that the present critical values $\beta_i$ are closely related to critical values given by the asymptotic optimal rejection curve which is obtained by Finner et al. \cite{AORC}. In the asymptotic set up they derived step up tests with asymptotic FDR control under various conditions. However, step up multiple tests given by the $\beta_i$'s do not control the FDR by the desired level $\alpha$ at finite sample size, see for instance Dickhaus \cite{Dickdiss}, Gontscharuk \cite{Gont_2010}.\newline
The intension of the present paper is twofold.
\begin{itemize}
\item We like to calculate the FDR of step down and step up tests more precisely using martingale and reverse martingale arguments. Here we get also new results under the basic independence model.
\item On the other hand we can extend the results for dependent p-values which are martingale or reverse martingale dependent. As application finite sample FDR formulas for step down and step up tests based on (\ref{betas}) are derived. We refer to the Appendix for a collection of examples of martingale models.
\end{itemize}
Martingale arguments were earlier used in Storey et al. \cite{stor_tayl}, Pena et al. \cite{pena}, Heesen and Janssen \cite{h_j_15.1} for step up and in Benditkis \cite{me} for step down multiple tests.\newline
This paper is organized as follows. Below the basic notations are introduced. Section \ref{sectionresultsd} presents our results for step down tests. A counterexample, Example \ref{exsec2}, motivates to study specific dependence concepts which allow FDR control, namely our martingale dependence model. The FDR formula, see (\ref{FDRgavrl}) below, consists of two terms. In particular, it relies on the expected number of false rejections which is studied in Sections \ref{sectionexp} and \ref{sectiondu}. Note that the results of Lemma \ref{nesandsuf} motivate naturally the consideration of martingale methods. Section \ref{sectionfdrco} is devoted to the FDR control under dependence which extends the results of Gavrilov et al. \cite{gavrilov_2009}. Within the class of step down tests the first coefficient $\beta_1$ is often responsable for the quality of the multiple test. In
Section \ref{sectionimpr} we propose an improvement of the power of SD procedures due to an increase of first critical values without loosing the FDR control.\newline
Step up multiple tests corresponding to the $\beta$'s from (\ref{betas}) are studied in Sections \ref{sectionsu} and \ref{sectionsufin}. We obtain the lower bound for the present FDR which can be greater than $\alpha.$ A couple of examples for martingale models can be found in Appendix. The proofs and additional material are collected in the Section \ref{sectionproofs}.\newline

{\bf Basics.}
Let us consider a multiple testing problem, which consists of $n$ null hypotheses 
$
H_1,...,H_n
$
with associated p-values $p_i, \ i=1,...,n.$ Assume that all p-values arise from the same experiment given by one data set, where each $p_i$ can be used for testing the traditional null $H_i$. The p-values vector
$
p=(p_1,...,p_n)\in[0,1]^n
$
is a random variable based on an unknown distribution $P.$ Recall that simultaneous inference can be established by so called multiple tests $\phi=\phi(p),$
$
\phi=(\phi_1,...,\phi_n):[0,1]^n\rightarrow \{0,1\}^n,
$
which rejects the null $H_i$ iff, i.e. if and only if, $\phi_i(p)=1$ holds.
The set of hypotheses can be divided in the disjoint union $I_0\bigcup I_1=\{1,...,n\}$ of unknown portions of true null $I_0$ and false null $I_1,$ respectively. We denote the number of true null by $n_0=|I_0|$ and the number of false ones by $n_1=|I_1|=n-n_0,$ where $n_0>0$ is assumed. Widely used multiple testing procedures can be represented as 
$$
\phi_{\tau}=(\1(p_1\leqslant \tau),...,\1(p_n\leqslant \tau))
$$ via the indicator function $\1(\cdot),$ where $\tau\in[0,1]$ is a random critical boundary variable. Thus all null hypotheses with related p-values that are not larger than the threshold $\tau$ have to be rejected. Let $p_{1:n}\leqslant p_{2:n} \leqslant \cdot\cdot\cdot \leqslant p_{n:n}$ denote the ordered values of the p-values $p$.
\begin{defi}
Let $\alpha_{1}\leqslant \alpha_{2}\leqslant \cdot\cdot\cdot\leqslant \alpha_{n} $
be a deterministic sequence of critical values. Set for convenience $\max\{\emptyset\}=0.$
\begin{itemize}
\item[(a)] The step down (SD) critical boundary variable is given by
\begin{align}
\tau_{SD}=\max\{\alpha_{i}:p_{j:n}\leqslant \alpha_{j}, \text{ \ for \ all \ } j\leqslant i\}.
\end{align}
\item[(b)] The step up (SU) critical boundary variable is given by
\begin{align}
\tau_{SU}=\max\{\alpha_{i}:p_{i:n}\leqslant \alpha_{i}\}.
\end{align}
\item[(c)] The appertaining multiple tests $\phi_{SD}=\phi_{\tau_{SD}}$ and $\phi_{SU}=\phi_{\tau_{SU}}$ are called step down (SD) test, step up  (SU) test, respectively.
\end{itemize}
\end{defi}
Let $\hat F_n$ denote the empirical distribution function of the p-values and let $V=V(\tau)=\sum\limits_{i\in I_0}\mathbb{I}(p_i\leqslant \tau)$, $S=S(\tau)=\sum\limits_{i\in I_1}\mathbb{I}(p_i\leqslant \tau)$ and  $R=R(\tau)=\sum\limits_{i=1}^{n}\mathbb{I}(p_i\leqslant \tau)=n\hat F_n(\tau)$ be the number $V$ of false rejections w.r.t. $\tau$, the number $S$ of true rejections and the number $R$ of all rejections, respectively.
The False Discovery Rate (FDR) of a procedure with critical boundary variable $\tau$ is defined as $$\text{FDR}=\E\left[\frac{V(\tau)}{R(\tau)}\right],$$
with the convention $\frac{0}{0}=0$.
The FDR is often chosen as an error rate control criterion. There is another useful equivalent description of step down tests.
\begin{remark}\label{rem.sigtau}
Introduce the random variable
\begin{align}\label{sig}
\sigma:=\min\{\alpha_i:p_{i:n}>\alpha_i\}\wedge \alpha_n,
\end{align}
where $a\wedge b=\min(a,b)$ denotes the minimum of two real numbers $a$ and $b.$
 
Then we have $\tau_{SD}\leqslant \sigma$ but the step down tests $\phi_{SD}=\phi_{\sigma}$ coincide and $\text{FDR}=\E\left[\frac{V(\tau_{SD})}{R(\tau_{SD})}\right]=\E\left[\frac{V(\sigma)}{R(\sigma)}\right]$ holds. The reason for this is that no p-value falls in the interval $(\tau_{SD},\sigma]$ and $R(\tau_{SD})=R(\sigma)$ is valid.
\end{remark}
 There is much interest in multiple tests such that the FDR is controled by a prespecified acceptable level $\alpha \in (0,1)$, i.e. to bound the expectation of the portion of false rejections.
The well known so called Benjamini and Hochberg multiple tests with linear critical values $\alpha_{i}=\alpha\frac{i}{n}$ lead to the FDR bound $$\text{FDR}\leqslant \alpha \frac{n_0}{n}$$ for SD and SU tests under positive dependence, more precisely under positive regression dependence on a subset (PRDS). There are several proposals to exhaust the FDR more accurate by $\alpha$ by an enlarged choice of critical values. A proper choice for SD tests are $\alpha_i$ 
\begin{align}\label{cvgavrl}
\begin{aligned}
&0<\alpha_{i}\leqslant \beta_{i}=\frac{i\alpha}{n+1-i(1-\alpha)},  \ 1\leqslant i\leqslant n,\\
&\alpha_{0}=\alpha_{1}, \ \beta_{0}=\beta_{1},
\end{aligned}
\end{align}
which allow the control $\text{FDR}\leqslant \alpha$ under the basic independence assumption of the p-values, see Gavrilov et al. \cite{gavrilov_2009}. Note that for $i=1,...,n,$ $\beta_{i}=g_{\alpha}^{-1}\left(\frac{i}{n}\right)$ are inverse values of
\begin{align}
g_{\alpha}(t)=\frac{n+1}{n}f_{\alpha}(t)=\frac{n+1}{n}\frac{t}{t(1-\alpha)+\alpha},
\end{align}
where $g_{\alpha}$ is close to the asymptotic optimal rejection curve $f_{\alpha},$ see Finner et al. \cite{AORC}. It is known that SU tests given by $\beta_{i}$ do not control the FDR for the independence model in general, see Gontscharuk \cite{Gont_2010}, Heesen and Janssen \cite{h_j_15.1}. If the p-values are dependent then the FDR control of the SD tests based on $\beta_{i}, \ i\leqslant n,$ can not be expected (see Example \ref{exsec2} of Section \ref{sectionresultsd}).\newline
Gavrilov et al. \cite{gavrilov_2009}, Theorem 1A, propose to reduce the critical values 
$\beta_{i}$ in order to get FDR control of SD-tests under positive regression dependence on a subset. Unfortunately, the procedure based on these new reduced critical values may be too conservative. Below we keep the critical values $\alpha_{i}, \ i\leqslant n,$ of (\ref{cvgavrl}) and introduce dependence assumptions for the p-values which insure the FDR-control for the underlying SD tests.\newline
The main idea of this paper can be outlined as follows. The FDR of SD and SU tests based on the critical values $\beta_{i}$ equals
\begin{align}\label{FDRgavrl}
\text{FDR}=\frac{\alpha}{n+1}\E\left[\frac{V}{\beta_{R}}\right]+\frac{1-\alpha}{n+1}\E\left[V\right].
\end{align}
A monotonicity argument implies the next Lemma.
\begin{lemma}\label{lemma1}
Consider an SD or SU test with critical values $(\alpha_{i})_i$ given by (\ref{cvgavrl}). Then
\begin{itemize}
\item[(a)]$\text{FDR}\leqslant \frac{\alpha}{n+1}\E\left[\frac{V}{\beta_{R}}\right]+\frac{1-\alpha}{n+1}\E\left[V\right]$.
\item[(b)] The conditions
\begin{align}
\label{FDR1}& \E\left[\frac{V}{\beta_{R}}\right]\leqslant n_0 \text{ \ and \ }\\
\label{FDR2}& \E\left[V\right]\leqslant \frac{\alpha}{1-\alpha}(n_1+1)
\end{align}
ensure the FDR control, i.e. FDR$\leqslant \alpha.$
\end{itemize}
\end{lemma}
Whereas the FDR is hard to bound under dependence, the inequality (\ref{FDR1}) is known under PRDS and equality holds under reverse martingale  structure (including the basic independence model), see Heesen and Janssen \cite{h_j_15.1} for SU test. Then it remains to bound the expected number of false rejections $\E\left[V\right]$, which is at least possible for SD tests under certain martingale dependence assumptions. In the following we always use a general assumption, that the p-values for the true null $(H_i)_{i\in I_0}$ fullfil
\begin{align}\label{genas}
\E\left[\sum\limits_{i\in I_0}\1(p_i\in [0,t])\right]\leqslant n_0 t \text{ \ for \ all \ }t\in[0,1),
\end{align}
which can be interpreted as ''stochastically larger'' condition compared with the uniform distribution in the mean for $I_0$. 
\newline

Now, we define the basic independence assumptions (BIA) that are often used in the FDR-control-framework.
\begin{itemize}
\item[(BIA)] We say that p-values fulfil the basic independence model if the vectors of p-values $\left(p_i\right)_{i\in I_0}$ and $\left(p_i\right)_{i\in I_1}$ are independent, and each dependence structure is allowed for the ``false" p-values within $\left(p_i\right)_{i\in I_1}.$
Under true null hypotheses the p-values $\left(p_i\right)_{i\in I_0}$ are independent and stochastically larger (or equal) compared to the uniform distribution on $[0,1],$ i.e., $P(p_i\leqslant x)\leqslant x$ for all $x\in[0,1]$ and $i\in I_0.$\newline
If in addition all p-values are i.i.d. uniformly distributed on $[0,1]$ for $i\in I_0$ then we talk about the BIA model with uniform true p-values. 
\end{itemize}
\section{Results for step down procedures}\label{sectionresultsd}
In this section we consider a step down procedure with critical values $\beta_{i}, \ i\leqslant n,$ from (\ref{betas}). It is well known that this procedure controls the FDR if the p-values fulfil the basic independence assumptions (BIA) (cf. Gavrilov et al. \cite{gavrilov_2009}). However, in practice the independence of the single tests corresponding to the present p-values are rare.

%
For general dependent p-values the FDR of the SD test may exceed the level $\alpha.$ The next counter example motivates the consideration of special kinds of dependence in order to establish FDR control.
\begin{example}\label{exsec2}
For $n=3, \ n_0=2, \ n_1=1, \ I_0=\{2,3\}$ and $\alpha=\frac{1}{4}$ consider the SD procedure with critical values $\beta_{i}=\frac{i\alpha}{n+1-i(1-\alpha)}$. Consider the vector of p-values $\left(0,U_1,U_2\right)$ with true p-values defined as follows
\begin{align*}
& U_1 \text{ \ is \ uniformly \ distributed \ on \ }(0,1).\\
& U_2=\left(U_1+\beta_{2}\right)\1\left(U_1\leqslant \beta_{2}\right)+\left(U_1-\beta_{2}\right)\1\left(U_1\in(\beta_{2},2\beta_{2})\right)+U_1\1\left(U_1\geqslant 2\beta_{2} \right).
\end{align*} 
For such p-values we get
\begin{align*}
\text{FDR}=\frac{2}{3}P(U_1\leqslant 2\beta_{2})=\frac{4\beta_{2}}{3}=\frac{4}{15}>\frac{1}{4}.
\end{align*}
\end{example}
We will start with the expected number of false rejections (ENFR), which was earlier studied by Finner and Roters \cite{Fi_Ro_2002} and Scheer \cite{Scheer}. 
\subsection{Control of the expected number of false rejections $\E\left[V\right]$}\label{sectionexp}
The present martingale approach relies on the empirical distribution function $\hat F_n$ of the p-values and on the adapted stochastic process 
\begin{align}\label{alphaproc}
\begin{aligned}
t\mapsto\hat\alpha(t)=\frac{t}{1-t}\frac{1-\hat F_n(t)}{\hat F_n(t)+\frac{1}{n}}, \ t\in T,
\text{w.r.t. \ the \ filtration \ }\\
\DF_t^{T}=\sigma\{\1(p_i\leqslant s),s\leqslant t, s,t\in T, \ i\leqslant n \} \text{ \ of \ the \ p-values}.
\end{aligned}
\end{align} 
Thereby, $T\subset [0,1)$ is a parameter space with $0\in T.$
The value $\hat\alpha(t)$ is frequentely used as a conservative estimator for the FDR on the constant critical boundary value $\tau=t.$ Storey et al. \cite{storey} use a similar estimator for the FDR$(t)$ of SU tests if the p-values are independent. A similar estimator is also used by Benjamini, Krieger and Yekutieli \cite{b_k_y_05}, Heesen and Janssen \cite{h_j_15.2} and Heesen \cite{hesdis}.
It is easy to see that for $\beta_{i}, \ i\leqslant n,$ we get from (\ref{cvgavrl})
\begin{align}\label{Lemma3}
&\hat\alpha_n(\beta_{i})\leqslant \alpha \text{ \ iff \ } R(\beta_{i})\geqslant i-1,\\
\label{12}&\hat\alpha_n(\beta_{i})= \alpha \text{ \ iff \ } R(\beta_{i})= i-1,
\end{align}
since
\begin{align*}
\hat \alpha_n(\beta_{i})=\alpha\left(\frac{i}{n+1-i}\right)\left(\frac{n-R(\beta_{i})}{R(\beta_{i})+1}\right).
\end{align*}
The consequences of these useful relations are summerized.
\begin{lemma}\label{eigen}
Consider the critical values $\left(\beta_i\right)_{i\leqslant n}$ and the critical boundary value $\sigma$ from (\ref{sig}). Then we have
\begin{itemize}
\item[(a)]$\sigma=\min\{\beta_i:\hat\alpha_n(\beta_i)\geqslant \alpha, i\leqslant n\}\wedge \beta_n.$
\item[(b)] Moreover $\tau_{SD}\leqslant \sigma$ and $\hat\alpha_n (\sigma)=\alpha\1(R(\sigma)<n)$ hold.
\item[(c)] The random variable $\sigma$ is a stopping time w.r.t. the filtration $\left(\DF_t^{T}\right)_{t\in T}$  of the p-values with time domain $T=\{0,\beta_1,...,\beta_n\}$.
\end{itemize}
\end{lemma}
It is quite obvious that the maximal coefficiens $\beta_i,i\leqslant n,$ of the $\alpha 's$ in (\ref{cvgavrl}) and the extreme p-values $p_i=0, \ i\in I_1,$ for all false null are least favourable for bounding $\E\left[V\right].$ First, we focus on the $\beta_i-$based SD procedure.
An important role plays the process
\begin{align}\label{mproc}
M_t=M_{I_0}(t)=\sum\limits_{i\in I_0} \frac{\1(p_i\leqslant t)-t}{1-t}, \ t\in T.
\end{align}
\begin{lemma}\label{nesandsuf}
Let $p_i=0 \text{ \ for \ all \ }i\in I_1.$ For the critical values $(\beta_i)_{i\leqslant n}$ from (\ref{betas}) we have 
\begin{align*}
\E\left[V(\tau_{SD})\right]\leqslant \frac{\alpha}{1-\alpha}(n_1+1)\text{ \ iff \ }
\E\left[M_{I_0}\right]\leqslant \alpha (n+1)P(R(\tau_{SD})=n).
\end{align*}
\end{lemma}
The probability $P(R(\tau_{SD})=n)$ is typically very small. Note that we will show below by martingale arguments that $\E\left[M_{I_0}(\tau_{SD})\right]\leqslant 0,$ which implies the crucial condition (\ref{FDR2}).

Next, we introduce a dependence assumption which allows the control of expected number of false rejection of the SD procedure with critical values $\beta_{i},i\leqslant n.$

\begin{itemize}
\item[(D1)] \label{md} Let $T\subset [0,1)$ be a set with $0\in T$.
We say that p-values $p_1,...,p_n$ are $\DF_T=\left(\DF_t\right)_{t\in T}-$ (super-) martingale dependent on a subset $J$ if the stochastic process $M(t)=M_{J}(t)=\sum\limits_{i\in J}^{}\left(\frac{\1(p_i\leqslant t)-t}{1-t}\right), t\in T,$ is a $\DF_T-$ (super-)martingale.
\end{itemize}
Note that the super-martingale model (D1) includes BIA if $J=I_0$. This is well known, see Shorack and Wellner \cite{ShoWell} (p. 133), Benditkis \cite{me}. Some examples of martingale dependent random variables can be found in a separate Appendix.\newline
Recall that the general condition (\ref{genas}) implies $\E\left[M_{I_0}(0)\right]=0,$ which is always assumed.\newline
The next remark shows that under (D1) we can assume that the p-values which belong to true null are stochastically larger compared with the uniform distribution on $(0,1)$ (cf. Heesen and Janssen \cite{h_j_15.1}, Benditkis \cite{me}). Let $U(0,1)$ denote the uniform distribution on the unit interval.
\begin{remark}\label{rem.martstochgr}
Let $\left(p_i\right)_{i\leqslant n}$ fulfil the martingale assumption (D1) for $J=I_0$ on $T\subset [0,1]$ and let $\sigma:I_0\rightarrow I_0$ be some random permutation of the index-set $I_0$ which is independent of $\left(p_i\right)_{i\leqslant n}$.
\begin{itemize}
\item[(a)] If
\begin{align*}
 M_{I_0}(t)=\sum\limits_{i\in I_0}\frac{\1(p_i\leqslant t)-t}{1-t} \text{ \ is \ a \ }\DF_T-\text{martingale,}
\end{align*}
then the random variable $Y_i=p_{\sigma(i)},i\in I_0,$ is U$(0,1)$-distributed. 
\item[(b)] If $\left(p_i\right)_{i\in I_0}$ fulfil the super-martingale assumption (D1) for $J=I_0$ on $T\subset [0,1]$, 
then $Y_i$, $i\in I_0,$ are stochastically larger compared with $U(0,1)$.
\item[(c)] As long as the boundary critical value $\tau$ only depends on the order statistics, the multiple test $\phi_{\tau}$ remains unchanged if $\left(p_i\right)_{i\leqslant n}$ is substituted by $\left((p_{\sigma(i)})_{i\in I_0},(p_i)_{i\in I_1}\right).$
\item[(d)] It can be shown that under (D1) the (super-)martingale assumption also holds under the filtration given by the exchangeable $\left(\left(Y_i\right)_{i\in I_0},\left(p_i\right)_{i\in I_1}\right).$ Note that the exchangeability of the $p_{\sigma(i)}, \ i\in I_0,$ is only needed in the proofs in connection with the PRDS assumption introduced in Section \ref{sectionfdrco}.
      \end{itemize}
\end{remark}
\begin{proof}[Proof of (a) and (b)]
 Firstly, note that the random variables $Y_i,i\in I_0$ are exchangeable, since $\sigma$ is an independent permutation. This implies
\begin{align}
 \E\left[\1(Y_i \leqslant t)\right]=\E\left[\frac{1}{n_0}\sum\limits_{i\in I_0}\1(Y_i \leqslant t)\right]=\E\left[\frac{1}{n_0}\sum\limits_{i\in I_0}\1(p_i \leqslant t)\right].
\end{align}
Moreover, we get 
\begin{align}\label{eqorin}
 \E\left[\frac{1}{n_0}\sum\limits_{i\in I_0}\1(Y_i \leqslant t)\right]=\frac{(1-t)}{n_0}\E\left[M_{I_0}(t)\right]+t\leqslant t.
\end{align}
Note that we have an (in)equality in (\ref{eqorin}), if $M_{I_0}(t)$ is a $\DF_T-$(super-)martingale.
\end{proof}
Now, we formulate the main result of this subsection under the super-martingale assumption, which will be applied to our equality (\ref{FDRgavrl}).
\begin{theorem}\label{ev}
Consider the SD multiple procedure with critical values $\beta_{i}$,  $i\leqslant n$, given in (\ref{cvgavrl}). Suppose that the super-martingale assumption (D1) holds with $J=I_0$ and $T=\{0,\beta_{1},...,\beta_{n}\}$.We get
\begin{align*}
\E\left[V(\tau_{SD})\right]\leqslant\frac{\alpha}{1-\alpha}(\E\left[S(\tau_{SD})\right]+1)\leqslant \frac{\alpha}{1-\alpha}(n_1+1).
\end{align*}
\end{theorem}

\subsection{Consequences under Dirac-Martingale configurations}\label{sectiondu}
In this subsection we consider the following assumptions
\begin{itemize}
\item[(i)] Martingale dependence assumption (D1) holds with $J=I_0 \ \text{and} \ T=\{0,\beta_1,...,\beta_n\}$,
\item[(ii)] $p_i=0$ a.s. for all $i\in I_1$.
\end{itemize}
Structures that fulfil the assumptions (i) and (ii) are called Dirac martingale configurations DM($n_1$).
The part (a) of the next lemma proposes exact formulas for ENFR for step down tests with critical values $\beta_{i}.$ Part (b) derives a lower bound for ENFR if the $(p_i)_{i\in I_1}$ are by accident uniformly distributed which is another example for extreme ordering compared with (ii).
\begin{lemma}[Some exact formulas for the ENFR]\label{exfor} Suppose that the martingale assumption (i) hold. Let $\tau_{SD}$ be the critical boundary value, which corresponds to critical values $\beta_{i}.$ \newline(a)  Assume additionally (ii) then
\begin{align}
             E_1:= \E_{\text{DM}(n_1)}\left[V(\tau_{SD})\right]=\frac{\alpha}{1-\alpha}(n_1+1)-\frac{\alpha}{1-\alpha}(n+1)P_{\text{DM}(n_1)}(V(\tau_{SD})=n_0)
             \end{align}
(b) Let $\left(p_1,...,p_n\right)$ be exchangeable and martingale dependent on $I_1$,i.e., $M_{I_1}(t)=\left(\frac{S(t)-n_1 t}{1-t}\right)_{t\in T}$ is an 
$\DF_T-$martingale. Then, each $p_i, \ i\leqslant n,$ is uniformly distributed and
\begin{align}
             E_2:=\E_{\text{U}(0,1)}\left[V(\tau_{SD})\right]=\frac{\alpha}{1-\alpha}\frac{n_0}{n}-\frac{\alpha}{1-\alpha}\frac{n+1}{n}P_{\text{U}(0,1)}(R(\tau_{SD})=n).
             \end{align}
(c) If $P(p_i\leqslant t)\geqslant t$ for all $i\in I_1$ and all $t\in[0,1]$ then 
\begin{align*}
E_2\leqslant \E\left[V(\tau_{SD})\right]\leqslant E_1.
\end{align*}
\end{lemma}
\subsection{Control of the FDR}\label{sectionfdrco}
As mentioned in Lemma \ref{lemma1} the control (\ref{FDR2}) of the ENFR is not enough for the FDR control. We have to bound $\E\left[\frac{V(\tau_{SD})}{\tau_{SD}}\right]$ by $n_0$. To do this we need further assumptions.
\begin{itemize}
\item[(D2)] \label{prds} The p-values are said to be positive regression dependent on a subset $J$ (PRDS) if
 $$x\mapsto\E\left[f(p_1,...,p_{n})\mid p_i=x\right]$$ 
is increasing in $x$ for each $i\in J$ and any coordinate-wise increasing, integrable function $f:[0,1]^n\rightarrow\mathbb{R}$ (cf. Finner et al. \cite{finnrot:2001}, Benjamini and Yekutielli \cite{ben_yek_01}.)
\end{itemize}
\begin{remark}
The assumption (D2) implies that $$x\mapsto\E\left[g(p_1,...,p_{n})\mid p_i\leqslant x\right]$$ 
is increasing in $x$ for each $i\in J$ and any coordinate-wise increasing, integrable function \\ $g:[0,1]^n\rightarrow\mathbb{R}$ (see Dickhaus \cite{Dickdiss}, Benditkis \cite{me}).
\end{remark}
The dependence assumption (D2) is well-known in the FDR-framework. Benjamini and Yekutielli \cite{ben_yek_01} proved that the Benjamini and Hochberg linear step up test controls the FDR under such kind of positive dependence. Gavrilov et al.\cite{gavrilov_2009} have shown that in this case the FDR of the step down procedure using critical values $\beta_i, \ i\leqslant n,$ may exeed the pre-chosen level $\alpha.$ Theorem \ref{fdrprds} proves the FDR control of that SD test under the additional super-martingale assumption.

\begin{theorem}\label{fdrprds}
Let $\left(p_i\right)_{i\in I}$ fulfil the super-martingale assumption (D1) with $T=\{0,\beta_1,...,\beta_m\}$ and the PRDS assumption (D2) on $I_0$. 
Then we have for $\beta_i-$based SD procedure
\begin{align*}
\text{FDR}_{\tau_{SD}}=\E\left[\frac{V(\tau_{SD})}{R(\tau_{SD})}\right]\leqslant \alpha.
\end{align*}
\end{theorem}
The next lemma is a technical tool for the proof of Theorem \ref{fdrprds}.
\begin{lemma}\label{ennul}
 Let $\left(p_i\right)_{i\leqslant n}$ fulfil (D2) on $I_0$. For the SD test based on the critical values $\beta_i,i\leqslant n,$ we have
\begin{align*}
\E\left[\frac{V(\tau_{SD})}{\tau_{SD}}\right]\leqslant n_0.
\end{align*}
\end{lemma}
\begin{remark}
\begin{itemize}
\item[(a)] Lemma \ref{ennul} remains true for any random variable $\tau=\tau(p),$ which is a non-increasing function of $p_i, \ i\in I_0,$ and has a finite range of values $\{a_1,...,a_m\}, \ 0<a_1\leqslant a_2\leqslant ...\leqslant a_m$ for some $m\in \mathbb{N}.$ That means that
\begin{align*}
\E\left[\frac{V(\tau)}{\tau}\right]\leqslant n_0
\end{align*}
can be always bounded under PRDS.
The exact structure of the random variable $\tau$ is not important. The inequality remains true for SD as well for SU tests.
\item[(b)] Theorem \ref{fdrprds} remains true for any $\DF_{\tilde T}-$ stopping time $\tilde \tau$ with $\tilde T=\{0,\tilde \beta_1,...,\tilde \beta_m\}, \ m\in \mathbb{N}, \ 0\leqslant\tilde\beta_1\leqslant...\leqslant\tilde\beta_m<1$, which is a non-increasing function of $p_i, \ i\in I_0$ if  $\tilde \tau\leqslant\sigma$ holds.
\end{itemize}
\end{remark}

The next theorem shows that we can relinquish the PRDS assumption (D2) under some modification of the assumption (D1).
\begin{theorem}\label{martcontr}
Let $M_{I_0}$ be a martingale w.r.t. to the new filtriration \newline$\DF_T^f=\sigma\left(\1(p_i\leqslant s), p_j, s\leqslant t, s\in T, i\in I_0,j\in I_1\right), \ t\in T,$ with $T=\{0,\beta_1,...,\beta_n\}$. Then, we get
\begin{align*}
\E\left[\frac{V(\tau_{SD})}{R(\tau_ {SD})}\right]\leqslant \E\left[\frac{V(\tau_{SD})}{S(\tau_ {SD})+1}\right]\leqslant\frac{\alpha}{1-\alpha}
\end{align*}
for the SD test based on $\beta_i,i\leqslant n$, that implies 
\begin{align*}
\E\left[\frac{V(\tau_{SD})}{R(\tau_ {SD})}\right]\leqslant\frac{\alpha}{1-\alpha}-\E\left[\frac{V(\tau_{SD})(V(\tau_{SD})-1)}{R(\tau_{SD})(S(\tau_{SD})+1)}\right].
\end{align*}
\end{theorem}
Observe that the filtrations $\DF_T$ and $\DF_T^{f}$ are different. The martingale condition w.r.t. $\DF_T^{f}$ holds if $M_{I_0}$ is a martingale conditioned under the outcomes $\left(p_i\right)_{i\in I_1}$, which is weaker than BIA. \\
Although the presented bound $\frac{\alpha}{1-\alpha}$ is slightly larger than $\alpha,$ the inequality can get a gain if the ratio $\frac{V(\tau_{SD})}{S(\tau_{SD})+1}$ is compared with the false discovery proportion $\frac{V(\tau_{SD})}{R(\tau_{SD})}.$
\subsection{Improvement of the power}\label{sectionimpr}
In this subsection we concentrate on the power of FDR-controlling procedures, which can be characterized by the value $\frac{\E\left[S(\tau)\right]}{n_1}$ for $n_1>0$. Let us consider a SD procedure with arbitrary critical values $\alpha_i, i\leqslant n,$ which controls the FDR. Then we can increase the corresponding critical boundary value $\tau_{SD}$ and improve the power of this procedure without loss of the FDR control under the PRDS assumption. Note that the result seems to be new also for the BIA model.
\begin{lemma}\label{firstcoef}
Assume the following:
\begin{enumerate}
  \item the random variables $\left(p_i\right)_{i\in\{1,...,n\}}$ satisfy (D2) on $I_0$,
  \item $\left(p_i\right)_{i\in I_0 }$ and $\left(p_i\right)_{i\in I_1}$ are stochastically independent,
  \item let each $p_i, i\in I_0$ be stochastically larger than $U(0,1),$
  \item the SD procedure using critical values $\alpha_i,i\leqslant n,$ controls the FDR at level $\alpha$ under 1.-3.
\end{enumerate}
Then a SD procedure using critical values $c_i=\max(\alpha_i,1-(1-\alpha)^{\frac{1}{n}}), \ i\leqslant n,$ controls the FDR at level $\alpha$.
\end{lemma} 
\begin{remark}
\begin{itemize}
\item[(a)]The critical value $1-(1-\alpha)^{\frac{1}{n}}$ is the smallest critical value of the SD procedure which was proposed by Benjamini and Liu \cite{benliu}. The procedure of Benjamini and Liu controls the FDR under BIA, see Benjamini and Liu \cite{benliu}, and under PRDS assumption, see Sarkar \cite{sarkar}.
\item[(b)]Due to Lemma \ref{firstcoef} we can increase the first critical value of the SD procedure based on the critical values $\alpha_i, i\leqslant n,$ from (\ref{cvgavrl}) in order to improve the power without loss of the FDR control.
\item[(c)] The critical values $c_i=\max(\beta_i,1-(1-\alpha)^{\frac{1}{n}})$ with $\beta_i, \ i\leqslant n$ from (\ref{cvgavrl}) are already larger than the critical values proposed by Benjamini and Liu \cite{benliu}.
\item[(d)] More general results about increased critical values can be found in Benditkis \cite{me}, Chap.4.
\end{itemize}
\end{remark}

If "false" p-values $f=(f_1,..,f_{n_1})=\left( p_i\right)_{i\in I_1}$ are specified then we denote the conditional expectation by $\E_f\left[\cdot\right]=\E\left[\cdot | (p_i)_{i\in I_1}\right].$
The next lemma is a simple consequence of Lemma \ref{firstcoef}.
\begin{lemma}\label{first}
Consider a SD procedure with arbitrary deterministic critical values 
$d_1,...,d_n$, $d_1\leqslant 1-(1-\alpha)^{\frac{1}{n}}$, and assume $n_1\geq 1.$ 
Let the following assumptions be fulfilled:
\begin{enumerate}
  \item the random variables $\left(p_i\right)_{i\in\{1,...,n\}}$ satisfy (D2) on $I_0$,
  \item $\left(p_i\right)_{i\in I_0 }$ and $\left(p_i\right)_{i\in I_1}$ are stochastically independent.
  \item $\E_{f_0}\left[\frac{V}{R}\right]\leqslant \alpha$ with $f_0=(0,f_2,...,f_{n_1})$ for all possible $f_2,...,f_{n_1}$ and all $n_0\in\mathbb{N}.$ 
\end{enumerate}
Then we have $\E\left[\frac{V}{R}\right]\leqslant \alpha$ for all possible random
$f=(f_1,...,f_{n_1})$. Thereby, $f=(f_1,...,f_{n_1})$ is the vector of the ordered p-values corresponding to false hypotheses.
\end{lemma}
\section{Results for SU Procedures}\label{sectionsu}
It is well known that the FDR of the SU multiple tests with critical values $\beta_i, \ i\leqslant n,$ see (\ref{FDRgavrl}), may exceed the prespecified level $\alpha.$ In particular, by  Lemma 3.25 of Gontscharuk \cite{Gont_2010} the worst case FDR is greater than $\alpha$ in the limit $n\rightarrow \infty.$ The reason for this is that $\E\left[V(\tau_{SU})\right]$ may exceed the bound $\frac{\alpha}{1-\alpha}(n_1+1)$ under some Dirac uniform configurations. Below the critical values $\beta_i$ are slightly modified in order to get finite sample FDR control. Main tools for the proof are reverse martingale arguments which were already applied by Heesen and Janssen \cite{h_j_15.1} for step up tests, which extend results for BIA models. Introduce the reverse filtration
\begin{align*}
\DG_t^T=\sigma((\1(p\leqslant s), 1\leqslant i \leqslant n,s\geqslant t),s,t\in T)
\end{align*}
given by the p-values. 
\begin{itemize}
\item [(R)]Let $T\subset (0,1]$ be a set with $1\in T$. We say that p-values $p_1,...,p_n$ are $\DG_t^T-$reverse super-martingale dependent if $\frac{V(t)}{t}=\frac{\sum\limits_{i\in I_0}\1(p_i\leqslant t)}{t}$ is a $\left(\DG_t^T\right)_{t\in T}$-reverse super-martingale. 
\end{itemize}

\begin{lemma}\label{X1}
Consider R-super-martingale dependent p-values for an index set $1\in T\subset [\delta,1]$ for some $\delta>0.$ Let $\tau$ be any $(\DG_t^T)_{t\in T}$ reverse stopping time with values $\tau$ in $T.$ Then we have
\begin{align}\label{x1}
\E\left[\frac{V(\tau)}{\tau}\right]\leqslant n_0
\end{align}
with equality "$=$" if the reverse super-martingale is a reverse-martingale.
\end{lemma} 
\begin{remark}
The inequality (\ref{x1}) is also fulfilled under the so called ''\textit{dependency control condition}'', which was proposed by Blanchard and Roquain \cite{Blanchard}. Note that the assumption (R) and the dependency control condition do not imply each other.
\end{remark}
Lemma \ref{X1} applies to various SU tests.
\begin{example}
Consider critical values $0<a_1\leqslant a_2\leqslant ...\leqslant a_n<1$ and an index set $T$, $\{a_1,a_2,...,a_n,1\}\subset T\subset [\delta,1]$ for some constant $0<\delta\leqslant a_1$.
\begin{itemize}
\item[(a)](SU tests given by $(a_i)_i$.) The variable 
\begin{align}\label{x2}
\tau=\max(a_i:p_{i:n}\leqslant a_i)\lor a_1
\end{align}
is a reverse stopping time with $\tau_{SU}\lor a_1=\tau, \ R(\tau_{SU})= R(\tau)$ and $V(\tau_{SU})=0$ if $\tau_{SU}\neq \tau.$ Thus
\begin{align}\label{x.3}
\E\left[\frac{V(\tau_{SU})}{\tau_{SU}}\right]=\E\left[\frac{V(\tau)}{\tau}\right]\leqslant n_0.
\end{align}
\item[(b)](Truncation of the SU test given by (a).) Assume the R-super-martingale condition for $T=[\eta,1]$ and $0<\eta\leqslant a_1$. Imagine that the statistician likes to reject
\begin{itemize}
\item at most k hypotheses, $1\leqslant k\leqslant n,$ but all $H_i$ with p-values $p_i\leqslant\eta.$ 
\item Introduce
$\tau_0=\max\{t\in[0,1]:\hat F_n(t)\leqslant \frac{k}{n}\}$ and the reverse stopping time
\begin{align}\label{x.4}
\tilde \tau=(\tau_0 \land \tau_{SU})\lor\eta.
\end{align}
Then, the inequality (\ref{x.3}) holds when $\tau$ is replaced by $\tilde \tau$. Obviously, also \\ $\E\left[V(\tilde\tau)\right]\leqslant\E\left[V(\tau)\right]$ follows. In case $a_i=\alpha_i,$ see (\ref{cvgavrl}), the condition $\E\left[V(\tau)\right]\leqslant \frac{\alpha}{1-\alpha}(n_1+1)$ thus, would imply control for $FDR_{SU}(\tau)$ as well as for $FDR_{SU}(\tilde\tau)$.
\end{itemize}
\end{itemize}
\end{example}
Below a finite sample exact SU multiple test under the BIA model is presented, which can be established by numerical methods or Monte Carlo tools. 
Consider new coefficients
\begin{align}\label{x5}
a_i=\frac{i\alpha}{n+1-i\delta}, \ i\leqslant n, \ 0\leqslant \delta<1-\alpha.
\end{align}
Let $\mathcal{P}_{BI(n)}$ denote all distributions of p-values at sample size $n$ under the basic independence BIA regime. Then, the worst case FDR of the step up test given by (\ref{x5}) under the parameters $(n,\delta)$ is
\begin{align}\label{x6}
\sup \limits_{\mathcal{P}_{BI(n)}} FDR(n,\delta)=\max\limits_{0\leqslant n_1<n}FDR_{DU(n_1)}(n,\delta)
\end{align}
given by a Dirac uniform configuration, where $FDR_{DU(n_1)}(n,\delta)$ denotes the step up FDR under DU$(n_1)$ with uniformly distributed p-values $p_i$ for $i\in I_0.$ Recall from Heesen and Janssen \cite{h_j_15.1} that there exists a unique parameter $\kappa_n=\delta\in (0,1-\alpha)$ for the coefficients  (\ref{x5}) with
\begin{align}\label{x7}
\sup\limits_{\mathcal{P}_{BI(n)}}FDR(n,\kappa_n)=\alpha
\end{align}
with larger (smaller) worst case FDR for $\delta>\kappa_n$ ($\delta<\kappa_n$, respectively). The solution $\kappa_n$ can be found by checking the maximum (\ref{x6}) of a finite number of constellations.\newline
The next theorem introduces the asymptotics of the present SU tests under the basic independence model. 
\begin{theorem}\label{x3}
Consider a sequence of SU tests with critical values $a_i=a_i(\delta_n),$ $1\leqslant i \leqslant n,$ given by (\ref{x5}) with $0\leqslant \delta_n<1-\alpha.$
\begin{itemize}
\item[(a)] Under the condition $\limsup\limits_{n\rightarrow \infty} \delta_n < 1-\alpha$ we have 
\begin{align}\label{x8}
\limsup\limits_{n\rightarrow \infty} \sup\limits_{\mathcal{P}_{BI(n)}} FDR(n,\delta_n)=\alpha.
\end{align}
\item[(b)] Assume that $\delta_n\rightarrow \delta\in(0,1-\alpha)$ and let the portion $\frac{n_0}{n}\leqslant c$ be limited by some constant $\alpha<c<1.$ Then 
\begin{align}\label{x9}
\limsup\limits_{n\rightarrow \infty} \sup\limits_{\mathcal{P}_{BI(n)}, \ n_0\leqslant cn} FDR(n,\delta_n)=\frac{cx(\delta)}{1-c+cx(\delta)}<\alpha,
\end{align}
where $x(\delta)=\frac{\left((c\alpha+\delta(1-c)-1)^2-4(1-c)c\alpha \delta\right)^{1/2}-c\alpha-\delta(1-c)+1}{2c\delta}.$
\end{itemize}
\end{theorem}
\begin{remark}
Theorem \ref{x3} together with the finite sample adjusted SU tests at parameter $\kappa_n,$ see (\ref{x7}), can be viewed as a finite sample contribution to the program of Finner et al. \cite{AORC}, who got the asymptotically optimal rejection curve for SU tests.
\end{remark}
\section{Finite results for SU tests using critical values of Gavrilov et al.}\label{sectionsufin}
Consider below a SU test using critical values $\beta_i=\frac{i\alpha}{n+1-i(1-\alpha)}, \ i\leqslant n.$ As mentioned above, this SU test may exceed the FDR level $\alpha$ under some Dirac uniform configurations (cf. Gontscharuk \cite{Gont_2010}, Heesen and Janssen \cite{h_j_15.1}).
\begin{center}
\begin{figure}[h]
\includegraphics[scale=0.33, bb=0 0 400 400]{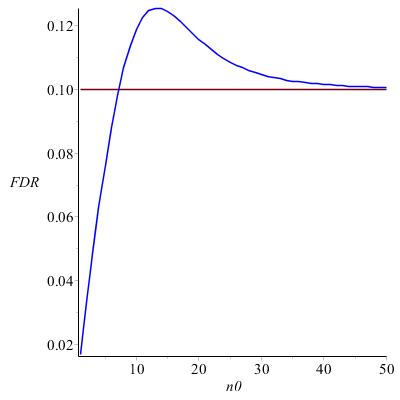}
\includegraphics[scale=0.33, bb=0 0 400 400]{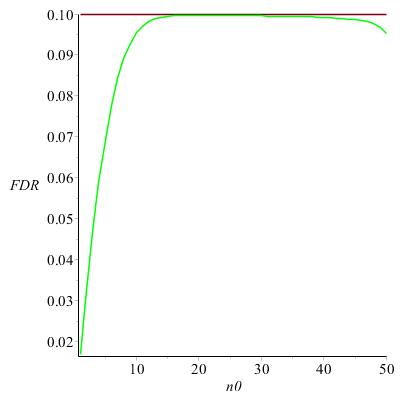}
\includegraphics[scale=0.33, bb=0 0 400 400]{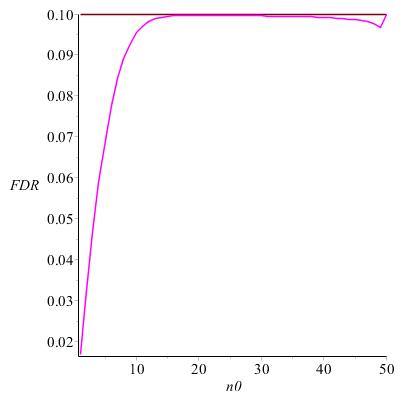}
\caption{{\footnotesize {The FDR under Dirac uniform configuration, as function of $n_0, \ 1\leqslant n_0\leqslant n$, of SU (blue line), SD (green line) and SD with improved first critical value (magenta line) for Dirac uniform configuration of the procedures based on the set of critical values $\beta_i, \ i\leqslant n$ with $n=50, \ \alpha=0.1.$ }}}\label{FDRDU}
\end{figure}
\end{center}
As we can see from the Figure \ref{FDRDU}, the FDR of the SU test may be larger than prechosen level $\alpha$ in contrast to the SD test based on the same critical values $\beta_i,i\leqslant n$. The next theorem gives an explanation in terms of the ENFR.
\begin{theorem}\label{X.5}
Let $(p_i)_{i\in I_0}$ fulfil the reverse martingale dependence assumption (R) and $p_j=0$ whenever $j\in I_1$. For $f(n)=\frac{2\alpha (n+1)^2}{n+3}$ we get 
\begin{itemize}
\item[(a)] $\E\left[V(\tau_{SU})\right]\geqslant \frac{\alpha}{1-\alpha}(n_1+1)$ for all $n_0\geqslant f(n),$
\item[(b)]FDR$_{\tau_{SU}}\geqslant \alpha$ for all $n_0\geqslant f(n)$.
\end{itemize}
Moreover, we have "$>$" in (a) and (b) if $n_0> f(n)$. 
\end{theorem}
In the concrete situation of Figure \ref{FDRDU} we observe $f(50)=9.8$, which is visible in the first grafic.
\section{The proofs and technical results}\label{sectionproofs}
The proof of Lemma 3 is obvious.
\begin{proof}[Proof of Lemma \ref{eigen}]
(a) Firstly, consider the case 
$\{\beta_i: p_{i:n}>\beta_i, \ i\leqslant n\}\neq \emptyset$ and define $j^*=\min\{i: p_{i:n}>\beta_i, \ i\leqslant n\}.$ Then, we get $\sigma=\beta_{j^*}$ due to the definition of $\sigma$. This implies 
\begin{align*}
\beta_{j^*}<p_{j^*:n} \text{\ and \ } \beta_{i}\geqslant p_{i:n} \text{ \ for \ all \ }i\leqslant j^{*}-1.
\end{align*}
Consequently we get 
\begin{align*}
p_{j^{*}-1}\leqslant \beta_{j^*-1}<\beta_{j^*}< p_{j^*:n},
\end{align*}
which implies $R(\beta_{j^*-1})=R(\beta_{j^*})=j^*-1$. Due to (\ref{12}) we get
\begin{align*}
j^{*}=\min\{i:\hat\alpha_n(\beta_i)=\alpha, \ i\leqslant n\},
\end{align*}
which completes the proof for this case. The case $\{\beta_i: p_{i:n}>\beta_i, \ i\leqslant n\}= \emptyset$ is obvious since $\hat\alpha(t)=0$ if $\hat F_n(t)=1$. 
\newline
The first statement of (b) is obvious and coincides with Remark 2. If there is any index $i\leqslant n$ with $R(\beta_i)=i-1$, then 
$\hat\alpha_n(\sigma)=\alpha$ due to (\ref{Lemma3}) and (a).\newline
Otherwise we have $ \hat\alpha_n(\beta_i)<\alpha$ for all $i\leqslant n$ and $R(\sigma)=R(\beta_{n})=n$ holds. This implies $\hat\alpha_n(\sigma)=0.$ Consequently, we get $\hat\alpha_n(\sigma)=\alpha\1(R(\sigma)<n)$.\newline
The part (c) is obvious.\end{proof}
Since $\tau_{SD}$ is not a stopping time w.r.t. $\DF_T$ we will turn to the critical boundary $\sigma$ in order to apply Lemma \ref{eigen}. 
\begin{proof}[Proof of Lemma \ref{nesandsuf} and Theorem \ref{ev}]
Firstly´, note that we have $V(\sigma)=V(\tau_{SD})$, as well as $S(\sigma)=S(\tau_{SD})$. Due to Lemma \ref{eigen} (b) we have $\hat\alpha_n(\sigma)=\alpha \1(R(\sigma)<n)$. Further, we obtain
\begin{align}\label{forexfor}
(1-\hat \alpha_n(\sigma))(R(\sigma)+1)=M_{I_0}(\sigma)+\frac{S(\sigma)-n_1 \sigma}{1-\sigma}+1
\end{align} 
by (\ref{alphaproc}) and (\ref{mproc}) which is a fundamental equation connection between $\hat \alpha (\cdot)$ and $M_{I_0}(\cdot)$.

In case $S(\sigma)=n_1$ of Lemma \ref{nesandsuf} we have
\begin{align}\label{imp1}
(1-\alpha)V(\sigma)=M_{I_0}(\sigma)+\alpha (n_1+1)-\alpha(n+1)\1(V(\sigma)=n_0),
\end{align}
which implies the equivalence in Lemma \ref{nesandsuf}.\newline
Under the conditions of Theorem \ref{ev} we have $\E\left[M_{I_0}(\sigma)\right]\leqslant 0$ by the optional sampling Theorem of stopped super-martingales. Thus, the fact that $\frac{S(\sigma)-n_1\sigma}{1-\sigma}\leqslant S(\sigma)$ implies 
\begin{align}
(1-\alpha)\E\left[V(\tau_{SD})\right]=(1-\alpha)\E\left[V(\sigma)\right]\leqslant \E\left[M_{I_0}(\sigma)\right]+\alpha\E\left[S(\sigma)+1\right]\leqslant \alpha(n_1+1)
\end{align}
due to (\ref{forexfor}) and Remark \ref{rem.sigtau}.
\end{proof}
\begin{proof}[Proof of Lemma \ref{exfor}]
(a) Consider again the equality (\ref{imp1}). If expectations are taken, the optional sampling theorem applies to $M_{I_0}(\sigma)$, which proves the result by Lemma \ref{eigen} (a).
\newline
(b) Analogous to the case (a) we have 
\begin{align}\label{exud}
 (1-\alpha)(V(\sigma)+S(\sigma)+1)+\alpha(n+1)\1(R(\sigma)=n)=M_{I_0}(\sigma)+M_{I_1}(\sigma)+1,
\end{align}
thereby
 $M_{I_1}(t)=\frac{\sum\limits_{i\in I_1}^{}\1(p_i\leqslant t)-n_1 t}{1-t}$ is a $\DF_t^T-$martingale.
Equality (\ref{exud}) implies
\begin{align}
 (1-\alpha)\E\left[V(\sigma)+S(\sigma)\right]=\alpha-\alpha(n+1)P(R(\sigma)=n)
\end{align}
by taking the expectation $\E$ and applying the Optional Sampling Theorem. The equality $\E\left[S(\sigma)\right]=\frac{n_1}{n_0}\E\left[V(\sigma)\right]$ (which follows from the assumption that all p-values are identically distributed) completes the proof of part (b) of this lemma. \newline
(c) The proof follows immediately from the observation that under martingale dependence the critical boundary value $\tau_{SD}$, and, consequently, $V(\tau_{SD})$, becomes maximal under assumptions of part (a) and minimal under assumptions of part (b). 
\end{proof}
\begin{remark}
\begin{enumerate}
\item The proof of the next Lemma \ref{ennul} uses the technique which was proposed by Finner and Roters \cite{finnrot:2001} for the proof of FDR-control of Benjamini and Hochberg test under PRDS.
\item  As long as we are vconcerned with the super-martingale assumption (D1) we may assume w.l.o.g. that $\left(p_i\right)_{ i\in I_0}$ are identically distributed and stochastically larger than $U(0,1)$, c.f Remark \ref{rem.martstochgr}. These technical tools are only used for the subsequent proofs of Sections 2.2 - 2.4 in connection with PRDS. The reference to Remark \ref{rem.martstochgr} is not cited again in each step of the proofs.
\end{enumerate}
\end{remark}

\begin{proof}[Proof of Lemma \ref{ennul}] Let us define $\beta_{0}=0$ for technical reasons and denote $(U_j)_{j\leqslant n_0}:=(p_i)_{i\in I_0}.$ Firstly, note that $\tau_{SD}=\beta_{R}$ holds obiously. Thereby, $R=R(\tau_{SD})$ is the number of rejections of the SD procedure with deterministic critical values $\beta_{i}, \ i\leqslant n$. We obtain the following sequence of (in)equalities:
\begin{align}
 &\E\left[\frac{V(\tau_{SD})}{\tau_{SD}}\right]=\sum\limits_{j=1}^{n_0}\E\left[\frac{\1(U_j\leqslant\beta_{R})}{\beta_{R}}\right]\\
  &=\sum\limits_{j=1}^{n_0}\sum\limits_{i=1}^{n}\E\left[\frac{\1(U_j\leqslant \beta_{i})}{\beta_{i}}\1(\beta_{R}=\beta_{i})\right]\\
 &=\sum\limits_{j=1}^{n_0}\sum\limits_{i=1}^{n}\E\left[\frac{\1(U_j\leqslant \beta_{i})}{\beta_{i}}\left(\1(\beta_{R}\leqslant \beta_{i})-\1(\beta_{R}\leqslant \beta_{i-1})\right)\right]\\
&= \sum\limits_{j=1}^{n_0}\sum\limits_{i=1}^{n}\frac{P(U_j\leqslant \beta_i)}{\beta_{i}}\E\left[\frac{\1(U_j\leqslant \beta_{i})\left(\1(\beta_{R}\leqslant \beta_{i})-\1(\beta_{R}\leqslant \beta_{i-1})\right)}{P(U_j\leqslant \beta_i)}\right]\\
\label{stovhgr}&\leqslant\sum\limits_{j=1}^{n_0}\sum\limits_{i=1}^{n}\left(\E\left[\1(\beta_{R}(p)\leqslant \beta_{i})|U_j\leqslant \beta_{i}\right]-\E\left[\1(\beta_{R}(p)\leqslant \beta_{i-1})|U_j\leqslant \beta_{i}\right]\right)\\
\label{boprds}&\leqslant \sum\limits_{j=1}^{n_0}\sum\limits_{i=1}^{n}\left(\E\left[\1(\beta_{R}\leqslant \beta_{i})|U_j\leqslant \beta_{i}\right]-E\left[\1(\beta_{R})\leqslant \beta_{i-1})|U_j\leqslant \beta_{i-1}\right]\right)\\
\label{telescoping}&= \sum\limits_{j=1}^{n_0} \E\left[\1(\beta_{R}\leqslant \beta_{n})|U_j\leqslant \beta_{n}\right] =n_0.
\end{align} 
The inequality in (\ref{stovhgr}) is valid since $U_1,...,U_{n_0}$ are stochastically greater than $U(0,1)$. The inequality in (\ref{boprds})
holds because the function $x\mapsto \1(\beta_{R}(p)\leqslant \beta_{i-1} \mid U_i\leqslant x)$ is increasing in $x$ for all $i\in\{1,...,n_0\}$ and since $U_1,...,U_{n_0}$ are assumed to be PRDS.
Consequently, using the telescoping sum we obtain the first equality in (\ref{telescoping}). The proof is completed because 
$\beta_{R}(p)\leqslant \beta_{n}$ by definition of $\beta_{R}=\tau_{SD}.$ 
\end{proof}
\begin{proof}[Proof of Theorem \ref{fdrprds}]
Combining Lemma \ref{lemma1}, Theorem \ref{ev} and Lemma \ref{ennul} yields the statement.
\end{proof}
To prove Theorem \ref{martcontr} we need the following technical result.
\begin{lemma}\label{martneg}
Under the assumptions of Theorem \ref{martcontr}
\begin{align*}
E\left[\frac{M(\sigma)}{S(\sigma)+1}\right]\leqslant 0.
\end{align*}
\end{lemma}
\begin{proof}[Proof of Lemma \ref{martneg}] First, note that the process $\beta_i\mapsto S(\beta_i)$ is always $\DF_T^f$-measurable. If we put $M=M_{I_0}$ then
\begin{align}\label{eqv}
\begin{aligned}
&\E\left[\frac{M(\sigma)}{S(\sigma)+1}\right]=\E\left[\sum\limits_{i=1}^{n}\frac{M(\beta_i)}{S(\beta_i)+1}\1(\sigma=\beta_i)\right]\\
&=\E\left[\sum\limits_{i=1}^{n}\frac{M(\beta_i)}{S(\beta_i)+1}\left(\1(\sigma\leqslant\beta_i)-\1(\sigma\leqslant \beta_{i-1} )\right)\right]\\
&=\sum\limits_{i=1}^{n}\left(\E\left[\frac{M(\beta_i)}{S(\beta_i)+1}\1(\sigma\leqslant\beta_i)\right]-\E\left[\E\left[\frac{M(\beta_i)}{S(\beta_i)+1}\1(\sigma\leqslant \beta_{i-1} )|\DF_{\beta_{i-1}}^f\right]\right]\right).
\end{aligned}
\end{align}
Since $\sigma$ is a $\DF_T^f$-stopping time, $S(\beta_i)$ is a measurable w.r.t. $\DF_{\beta_{i-1}}^f$ and using martingale property $\E\left[M(\beta_i)|\DF_{\beta_{i-1}}^f\right]=M(\beta_{i-1})$ we get
\begin{align*}
\E\left[\E\left[\frac{M(\beta_i)}{S(\beta_i)+1}\1(\sigma\leqslant \beta_{i-1} )|\DF_{\beta_{i-1}}^f\right]\right]=\E\left[\frac{M(\beta_{i-1})}{S(\beta_i)+1}\1(\sigma\leqslant \beta_{i-1})\right].
\end{align*}
Define $\beta_0=0.$ Now, we can continue the chain of equalities (\ref{eqv}) as follows.  
\begin{align}
\begin{aligned}\label{firstsum}
&\sum\limits_{i=1}^{n}\left(\E\left[\frac{M(\beta_i)}{S(\beta_i)+1}\1(\sigma\leqslant\beta_i)\right]-\E\left[\frac{M(\beta_{i-1})}{S(\beta_{i-1})+1}\1(\sigma\leqslant\beta_{i-1})\right]\right)\\
&+\sum\limits_{i=1}^{n}\left(\E\left[\frac{M(\beta_{i-1})}{S(\beta_{i-1})+1}\1(\sigma\leqslant\beta_{i-1})\right]-\E\left[\frac{M(\beta_{i-1})}{S(\beta_i)+1}\1(\sigma\leqslant \beta_{i-1})\right]\right)\\
\end{aligned}
\end{align}
\begin{align}
&=\sum\limits_{i=1}^{n}\left(\E\left[\frac{S(\beta_i)-S(\beta_{i-1})}{(S(\beta_{i-1})+1)(S(\beta_{i}+1))}\E_f\left[M(\beta_{i-1})\1(\sigma\leqslant \beta_{i-1}))\right]\right]\right)
\end{align}
because the first term in (\ref{firstsum}) is equal to zero due to the telescoping sum since $\E\left[M(\beta_n)\right]=0$.
Now, we will show that 
\begin{align}\label{minus}\E\left[M(\beta_{i-1})\1(\sigma\leqslant \beta_{i-1}))\right]\leqslant 0
\end{align} for all $i\leqslant n.$ Indeed, we have
\begin{align}
\E\left[M(\beta_{i-1})\1(\sigma\leqslant \beta_{i-1}))\right]=-\E\left[M(\beta_{i-1})\1(\sigma> \beta_{i-1}))\right].
\end{align}
Further, by the definition of $\sigma$ we know that $\hat\alpha_n(t)\leqslant \alpha$ for all $t\leqslant \sigma, \ t\in T$. Hence, due to (\ref{Lemma3}) we get $\1(\sigma> \beta_{i-1}))=\prod\limits_{j=0}^{i-1}\1\left(R(\beta_j)\geqslant j-1\right)$. On the other hand, we can conclude from the definition of the process $M_{I_0}$ that $\1\left(R(\beta_j)\geqslant j-1\right)=\1\left(M(\beta_j)\geqslant c(j)\right)$. Thereby, constants $c_j, \ j\leqslant i-1,$ are defined as
$c_j=\frac{j-1-S(\beta_j)-n_0\beta_j}{1-\beta_j}.$ Consequently, (\ref{minus}) is equivalent to 
\begin{align}
\E\left[M(\beta_{i-1})\prod\limits_{j=0}^{i-1}\1\left(M(\beta_j)\geqslant c(j)\right))\right]\geqslant 0,
\end{align}
which follows immediately from the following Lemma \ref{1mart} (by setting $X=M(\beta_{i-1})$ and $A=\prod\limits_{j=0}^{i-2} \1(M(\beta_{j})\geqslant c(j))$) and Lemma \ref{2mart}.
\end{proof}
\begin{lemma}\label{1mart}
Let $X$ be a random variable with $\E\left[X\right]=0, \ A$ be a measurable set and $c\in  \mathbb{R}$ be some constant. The inequality $\E\left[\1_{A}X\right]\geqslant 0$ implies $\E\left[\1_{A}X\1(X>c)\right]\geqslant 0.$
\end{lemma}
\begin{proof}[Proof of Lemma \ref{1mart}]
The case $c\geqslant 0$ is obvious. If $c<0$ we get
\begin{align*}
0\leqslant \E\left[\1_{A}X\right]= \E\left[\1_{A}X\1(X>c)\right]+\E\left[\1_{A}X\1(X\leqslant c)\right]\leqslant \E\left[\1_{A}X\1(X>c)\right],
\end{align*}
which implies the proof of this lemma.
\end{proof}
\begin{lemma}\label{2mart}
Under the assumptions of Theorem \ref{martcontr} we have
\begin{align}\label{toprove}
\E\left[M(\beta_{i-1})\prod\limits_{j=0}^{i-2} \1(M(\beta_{j})\geqslant c(j))\right]\geqslant 0 \ \text{for} \ i\geqslant 2.
\end{align}

\end{lemma}
\begin{proof}[Proof of Lemma \ref{2mart}]
The proof is based on induction. Firstly, we define $k:=i-1$.\newline
Let $k=2$ then ($\ref{toprove}$) is equivalent to
\begin{align*}
\E\left[M(\beta_{2})\1(M(\beta_{1})\geqslant c(1))\right]\geqslant 0,
\end{align*}
which is true due to the following chain of equalities.
\begin{align}\label{indan}
\begin{aligned}
&\E\left[M(\beta_{2})\1(M(\beta_{1})\geqslant c(1))\right]=\E\left[\E\left[M(\beta_{2})\1(M(\beta_{1})>c(1))\geqslant c(1))|\DF^f_{\beta_1}\right]\right]\\
&=\E\left[\1(M(\beta_{1})\E\left[M(\beta_{2})|\DF^f_{\beta_1}\right]\right]=\E\left[M(\beta_{1})\1(M(\beta_{1})\geqslant c(1))\right].
\end{aligned}
\end{align}
The two last equalities in (\ref{indan}) are valid due to the measurability of $M(\beta_{1})$ w.r.t. $\DF^f_{\beta_1}$ and martingale property $\E\left[M(\beta_{2})|\DF^f_{\beta_1}\right]=M(\beta_{1}).$\newline
Assume that
\begin{align}\label{stepk}
\E\left[M(\beta_{k})\prod\limits_{j=0}^{k-1} \1(M(\beta_{j})\geqslant c(j))\right]\geqslant 0
\end{align}
holds for $k\geqslant 2.$ Then we prove that
\begin{align*}
\E\left[M(\beta_{k+1})\prod\limits_{j=0}^{k} \1(M(\beta_{j})\geqslant c(j))\right]\geqslant 0
\end{align*}
is also true. We have
\begin{align*}
&\E\left[M(\beta_{k+1})\prod\limits_{j=0}^{k} \1(M(\beta_{j})\geqslant c(j))\right]=\E\left[\E\left[M(\beta_{k+1})\prod\limits_{j=0}^{k} \1(M(\beta_{j})\geqslant c(j))|\DF^f_{\beta_{k}}\right]\right]\\
&=\E\left[M(\beta_{k})\prod\limits_{j=0}^{k} \1(M(\beta_{j})\geqslant c(j))\right]=\E\left[\left(M(\beta_{k})\prod\limits_{j=0}^{k-1} \1(M(\beta_{j})\geqslant c(j))\right)\1(M(\beta_{k})\geqslant c(k))\right]\geqslant 0.
\end{align*}
The last inequality follows from (\ref{stepk}) and Lemma  \ref{1mart}.
\end{proof}
Now, we are able to prove Theorem \ref{martcontr}.
\begin{proof}[Proof of Theorem \ref{martcontr}]
Let us remind (\ref{forexfor}), which implies
\begin{align}
(1-\alpha)V(\sigma)\leqslant M_{I_0}(\sigma)+\alpha(S(\sigma)+1).
\end{align}
Dividing by $S(\sigma)+1$ yields
\begin{align*}
(1-\alpha)\frac{V(\sigma)}{S(\sigma)+1}\leqslant \frac{M_{I_0}(\sigma)}{S(\sigma)+1}+\alpha.
\end{align*}
Taking the expectation $\E\left[\cdot\right]$ and using Lemma \ref{martneg} deliver the result.
\end{proof}
To prove Lemma \ref{firstcoef} we need the following technical result.
\begin{lemma}\label{pod}
Let $(b_i)_{i\in I_0}$ be some set of real numbers with $b_i\in[0,1], \ i\in I_0$.
If p-values $(p_i)_{i\leqslant n}$ are PRDS on $I_0$ then 
\begin{align}
P\left(\bigcap\limits_{i\in I_0}\{ p_i> b_i\}\right)\geqslant\prod\limits_{i\in I_0}P\left(p_i>b_i\right).
\end{align}
\end{lemma}
\begin{proof}
W.l.o.g. let us assume that $I_0=\{1,...,n\}$. For any other subset $I_0$ the proof works in the same way.
First, we show  the inequality
\begin{align} \label{prdspqd}
\E\left[\prod\limits_{i=2}^{n}\1(p_i>b_i)\mid p_1>b_1 \right]\geqslant\E\left[\prod\limits_{i=2}^{n}\1(p_i>b_i)\mid p_1\leqslant b_1 \right].
\end{align}
To do this let $F$ denote the marginal distribution function of $p_1$. 
Define $f(u)=\E\left[\prod\limits_{i=2}^{n} \1(p_i>b_i)|p_1=u\right]$, then we have 
$$\E\left[\prod\limits_{i=2}^{n}\1(p_i>b_i)\1( p_1>b_1) \right]=\int \limits_{b_1}^{1}f(u)dF(u) $$
and
$$ \E\left[\prod\limits_{i=2}^{n}\1(p_i>b_i)\1( p_1\leqslant b_1) \right]=\int \limits_{0}^{p_1}f(u)dF(u).$$
Then (\ref{prdspqd}) is equivalent to
\begin{align}\label{prdspqd1}
 \frac{\int \limits_{b_1}^{1}f(u)dF(u)}{1-F(b_1)}\geqslant\frac{\int \limits_{0}^{b_1}f(u)dF(u)}{F(b_1)}.
\end{align}
From the mean value theorem for Riemann-Stieltjes integrals we can deduce that there exist some values $\xi_1$ and $\xi_2$ with
\begin{align}\begin{aligned}
 &\xi_1=\frac{\int\limits_{b_1}^{1}f(u)dF(u)}{1-F(b_1)},&\mbox{\ \ \ } &\inf\limits_{t\in(b_1,1)}f(t)\leqslant\xi_1\leqslant \sup\limits_{t\in(b_1,1)}f(t),\\
  &\xi_2=\frac{\int\limits_{0}^{b_1}f(u)dF(u)}{F(b_1)},  &\mbox{\ \ \ }&\inf\limits_{t\in(0,b_1)}f(t)\leqslant\xi_2\leqslant \sup\limits_{t\in(0,b_1)}f(t).\label{xi12}
\end{aligned}
\end{align}
Since $f$ is an increasing function of $u$, (\ref{xi12}) yields $\xi_1\geqslant \xi_2,$ hence we get (\ref{prdspqd1}).\newline
Further, we obtain
 \begin{align}
& P(p_1>b_1,p_2>b_2,...,p_{n}>b_{n})=\E\left[\prod\limits_{i=1}^{n}\1(p_i>b_i)\right] \\
\label{prds*} &=P(p_1>b_1)\E\left[\prod\limits_{i=2}^{n}\1(p_i>b_i)\mid p_1>b_1 \right]\geqslant  P(p_1>b_1) \E\left[\prod\limits_{i=2}^{n}\1(p_i>b_i)\right]\\
\label{prds1}&=P(p_1>b_1) P(\bigcap\limits_{i=2}^{n}(p_i>b_i))\geqslant ... \geqslant \prod\limits_{i=1}^{n}P(p_i>b_i).
   \end{align}
The inequality in (\ref{prds1}) holds due to the PRDS-assumption since according to the law of total probability we have
\begin{align}\begin{aligned}
 &\E\left[\prod\limits_{i=2}^{n} \1(p_i>b_i)\right]=\E\left[\prod\limits_{i=2}^{n}\1(p_i>b_i)\mid p_1>b_1 \right] -\\ 
&P(p_1\leqslant b_1)\underbrace{\left(\E\left[\prod\limits_{i=2}^{n}\1(p_i>b_i)\mid p_1>b_1 \right]-\E\left[\prod\limits_{i=2}^{n}\1(p_i>b_i)\mid p_1\leqslant b_1 \right]\right)}_{\geqslant 0 \text{ \ due \ to }(\ref{prdspqd})}.
\end{aligned}
\end{align}
\end{proof}
For the two following proofs we define the p-values, which belong to the true null by $\left(U_i\right)_{i\in\{1,...,n_0\}}=\left(p_j\right)_{j\in I_0}$.
Now, we are able to prove Lemma \ref{firstcoef} by conditioning under the portion $f$ belonging to the false null.
\begin{proof}[Proof of Lemma \ref{firstcoef}]
We consider an arbitrary FDR-controlling SD procedure that uses critical values $\alpha_i,i\leqslant n.$
Let us define $j^*:=\max\{i: f_j\leqslant 1-(1-\alpha)^{\frac{1}{n}} \text{\ for \ all \  }j\leqslant i\}$ and consider two possible cases.
\begin{itemize}
\item[1.] Let $j^*=0$. In this case we have under PRDS assumption
\begin{align*}
&\E_f\left[\frac{V}{R}\1(V>0)\right]\leqslant \E_f\left[\1(V>0)\right]=P(U_{1:n_0}\leqslant 1-(1-\alpha)^{\frac{1}{n}} )\\
&=1-P(\bigcap\limits_{i=1}^{n_0}\{U_i>1-(1-\alpha)^{\frac{1}{n}}\})\leqslant \alpha,
\end{align*} 
where the last inequality is valid due to Lemma \ref{pod}.
\item[2.]Let  $j^*>0$. Define the vector 
$f^*_{0}=(0,...,0,f_{j^*+1},...,f_{n_1}),$ where the $j^*$ first coordinates are replaced by $0.$ We get
\begin{align}
\E_f\left[\frac{V}{R}\1(V>0)\right]\leqslant \E_f\left[\frac{V}{j^*+V}\right]= \E_{f^*_{0}}\left[\frac{V(\tau_{SD})}{R(\tau_{SD})}\right]\leqslant \alpha.
\end{align}
Thereby, $\tau_{SD}$ is the critical boundary value corresponding to the SD procedure with critical values $\alpha_i, \ i\leqslant n$.
\end{itemize}
\end{proof}
\begin{proof}[Proof of Lemma \ref{first}]
 Note that $\E\left[\frac{V}{R}\right]=\E\left[\1(f_1\leqslant d_1)\frac{V}{R}+\1(f_1>d_1)\frac{V}{R}\right]$ is always valid and let us consider two different cases: (a) $f_1\leqslant d_1$, (b) $f_1>d_1$. \newline
(a) Since $f_1$ will be rejected, the equality $\E_{f_0}\left[\frac{V}{R}\right]=\E_f\left[\frac{V}{R}\right]$ holds. Therefore the statement of the lemma is proved for this case.\newline
(b) If $f_1>d_1$  holds, we have due to the PRDS assumption that
\begin{align*}
 \E_f\left[\frac{V}{R}\1(V>0)\right]&\leqslant \E_f\left[\1(V>0)\right]=P(U_{1:n_0}\leqslant d_1)\\
&\leqslant P(U_{1:n_0}\leqslant 1-\sqrt[n]{1-\alpha})\leqslant \alpha, \mbox{ \ \ \ } 
\end{align*}
where $U_{1:n_0}$ is the smallest true p-value.
Hence, we get $\E_f\left[\frac{V}{R}\right]\leqslant \alpha$ for all possible vectors $f=(f_1,...,f_{n_1})$. Further, 
 $\E\left[\frac{V}{R}\right]=\E\left[\E_f\left[\frac{V}{R}\right]\right],$ which completes the proof. 
\end{proof}

\begin{proof}[Proof of Lemma \ref{X1}]
The optional stopping theorem for reverse martingales implies
\begin{align}\label{}
\E\left[\frac{V(\tau)}{\tau}\right]\leqslant \frac{V(1)}{1}=n_0.
\end{align}
In the case of reverse martingales we have an equality.
\end{proof}
It is quite obvious that the variables (\ref{x2}) and (\ref{x.4}) are stopping times and Lemma \ref{X1} can be applied.\newline
The proof of Theorem \ref{x3} requires some preparations. Consider a wider class of rejection curves given by positive parameters $b$ and $\alpha, \ \delta\geqslant 0$ and
\begin{align}\label{Y2}
g(t)=\frac{tb}{\delta t + \alpha}, \ 0\leqslant t \leqslant 1, \ b>\delta+\alpha.
\end{align}
Note that the condition $g(1)>1$ is necessary for proper SU tests with critical values
\begin{align}\label{Y3}
a_i:=g^{-1}\left(\frac{i}{n}\right)=\frac{\alpha i}{nb-i\delta}.
\end{align}
By the choice $b=\frac{n+1}{n}$ and $\delta=1-\alpha$ the coefficients $\beta_i$ are included. Thus we arrive at the following equation for the FDR of (\ref{Y3})
\begin{align}\label{Y4}
\text{FDR}=\E\left[\frac{V}{R}\right]=\frac{\alpha}{nb}\E\left[\frac{V}{a_R}\right]+\frac{\delta}{nb}\E\left[V\right]
\end{align}
for each multiple test. In contrast to SD tests the term $\E\left[\frac{V}{a_R}\right]$ can be bounded under the R-super-martingale condition, cf. Heesen and Janssen \cite{h_j_15.1}.
\begin{remark}\label{Y1}
Consider SU tests for parameters $(\delta,b,\alpha)$ under R-super-martingale models with fixed portion $n_1<n$ of false  p-values.
\begin{itemize}
\item[(a)] The R-martingale models are least favourable for bounding $\E\left[\frac{V}{a_R}\right].$
\item[(b)]Consider two settings $\left(p_i\right)_{i\leqslant n}$
and $\left(q_i\right)_{i\leqslant n}$ of R-martingale models. If $p_i\leqslant q_i$ holds for all $i\in I_1$ then $\E_p\left[V\right]\leqslant \E_q\left[V\right]$ and $\text{FDR}_p\leqslant\text{FDR}_q$ holds.
\item[(c)] Under BIA with uniformly distributed p-values under the null, as well as under the reverse martingale dependence (R), we have:
\begin{itemize}
\item[(i)] The Dirac uniform configuration DU$(n_1)$ is least favourable for $\E\left[V\right]$ and FDR.
\item[(ii)] Suppose that $\frac{\alpha}{b}$ or $\frac{\delta}{b}$ increases. Then the coefficients $a_i$ increase and the FDR and $\E\left[V\right]$ do not decrease.
\end{itemize}
\end{itemize}
\end{remark}
\begin{proof}[Proof of Theorem \ref{x3}](a)
Proposition 4.1
of Heesen and Janssen \cite{h_j_15.1} establishes the asymptotic lower bound:
\begin{align*}
\sup\limits_{\mathcal{P}_{BI(n)}}\text{FDR}(n,\delta_n)\geqslant \min\limits_{i\leqslant n}\frac{na_i}{i}=\frac{n\alpha}{n+1-\delta_n}\rightarrow \alpha
\end{align*}
as $n\rightarrow \infty.$ To obtain the upper bound we can first exclude all coefficients $\delta_n=0$, which correspond to a Benjamini and Hochberg test with level $\frac{n\alpha}{n+1}.$ Fix some value $\gamma$ with $\limsup\limits_{n\rightarrow \infty} \delta_n<\gamma<1-\alpha$ and introduce the rejection curve $g_{\gamma}(t)=\frac{t}{\gamma t+\alpha}.$
For all $\delta_n<\gamma$ the FDR$(n,\delta_n)$ of the $a_i$'s can now be compared with the FDR of the SU test with critical values $g_{\gamma}^{-1}\left(\frac{i}{n}\right)$. By (\ref{Y4}) and Remark (\ref{Y1}) we have for each regime
\begin{align*}
\text{FDR}(n,\delta_n)\leqslant \text{FDR}_n(g_{\gamma}^{-1})
\end{align*}
using obvious notations. The worst case asymptotics is given by Theorem 5.1 of Heesen and Janssen \cite{h_j_15.1}
\begin{align*}
\limsup\limits_{n\rightarrow\infty}\sup\limits_{\text{BI}(n)}\text{FDR}_n(g_{\gamma}^{-1})=K,
\end{align*}
where 
\begin{align*}
K=\sup\left\{\frac{x}{1-x}\frac{1-g_{\gamma}(x)}{g_{\gamma}(x)}:0<x\leqslant g_{\gamma}^{-1}(1)\right\}=\alpha.
\end{align*}
(b) Similarly as above the FDR$(n,\delta_n)$ is bounded below and above by the FDR of SU given by rejection curves. Choose constants $0<\gamma_1<\delta<\gamma_2<1-\alpha$ and $b>1$ and consider $\delta_n\in(\gamma_1,\gamma_2)$ and large $n$ with $\frac{n+1}{n}\leqslant b.$ Introduce $g_{\gamma_2}(t)=\frac{t}{\gamma_2t+\alpha}$ and  $g_{\gamma_1,b}(t)=\frac{tb}{\gamma_1t+\alpha}$. Again we have
\begin{align*}
\text{FDR}_n(g_{\gamma_1,b})\leqslant\text{FDR}(n,\delta_n)\leqslant \text{FDR}_n(g_{\gamma_2}).
\end{align*}
Let $x(\gamma_2)\in(0,1)$ denote the unique solution of the equation
\begin{align}\label{Y5}
g_{\gamma_2}(x)=(1-c)+xc.
\end{align}
If we repeat the proof of Proposition 5.1 of Heesen and Janssen \cite{h_j_15.1} we have
\begin{align*}
\limsup\limits_{n\rightarrow\infty}\text{FDR}_n(g_{\gamma_2})=\sup\{\frac{x}{1-x}\frac{1-g_{\gamma_2}(x)}{g_{\gamma_2}(x)}:x(\gamma_2)\leqslant x\leqslant g_{\gamma_2}^{-1}(1)\}
\end{align*}
which is equal to $\frac{cx(\gamma_2)}{(1-c)+cx(\gamma_2)}.$ Similarly, a lower bound of (\ref{x9}) is $\frac{cx(\gamma_1,b)}{(1-c)+cx(\gamma_1,b)}$ with solution $x(\gamma_1,b)\in(0,1)$ of (\ref{Y5}). If now $\gamma_1 \uparrow \delta, \ \gamma_2\downarrow \delta$ and $b\downarrow\delta,$ the bounds turn to the value $\frac{cx(\delta)}{1-c+cx(\delta)}$ given by the solution $x(\delta)\in(0,1)$ of (\ref{Y5}).
\end{proof}
\begin{proof}[Proof of Theorem \ref{X.5}]
Define the process $\left(\tilde M(t)\right)_{t\in T}=\left(\frac{V(t)}{t}-n_0\right)_{t\in T},$ which is, obviously, a centered reverse martingale w.r.t $\DG_t^T$ due to the reverse martingale assumption. Now, we remind that for the step-wise procedure using critical values $(\beta_i)_{i\leqslant n}$ the following equality is valid by (\ref{imp1}):
\begin{align*}
(1-\alpha)V(\tau_{SU})=\tilde M(\tau_{SU})\frac{\tau}{1-\tau}-\alpha(n+1)\1(V(\tau_{SU}=n_0))+\alpha(n_1+1)
\end{align*}
under the Dirac distribution of "false" p-values $(p_i)_{i\in I_1}.$ Thus, we have to show  
$$\E\left[\tilde M(\tau_{SU})\frac{\tau_{SU}}{1-\tau_{SU}} \right]\geqslant \alpha (n+1)P(V(\tau_{SU})) \text{ \ for \ } n_0\geqslant f(n)$$ 
to prove the part (a). First note, that because of $V(\tau_{SU})=V(\tau_{SU}\vee \beta_{1})$ it is enough to prove the statement of this theorem for the reverse stopping time $\tilde \tau_{SU}=\tau_{SU}\vee \beta_{1}$, where  $\tilde\tau_{SU}\in\{\beta_1,...,\beta_n\}.$ Fix an $\varepsilon=\varepsilon(n)>0$ with $\beta_n+\varepsilon<1$ and define a fictive additional coefficient $\beta_{n+1}=\beta_n+\varepsilon.$ Then, we have
\begin{align}\label{rmenfr}
&\E\left[\tilde M(\tilde\tau_{SU})\frac{\tilde\tau_{SU}}{1-\tilde\tau_{SU}}\right]=\sum\limits_{i=1}^{n}\E\left[\tilde M(\beta_i)\frac{\beta_i}{1-\beta_i}\1(\tilde\tau_{SU}=\beta_i)\right]\\
\label{rmenfr1}&=\sum\limits_{i=1}^{n}\E\left[\tilde M(\beta_i)\frac{\beta_i}{1-\beta_i}\1(\tilde\tau_{SU}\geqslant \beta_i)-\tilde M(\beta_i)\frac{\beta_i}{1-\beta_i}\1(\tilde\tau_{SU}\geqslant \beta_{i+1})\right].
\end{align}
Consider the term $\E\left[\tilde M(\beta_i)\frac{\beta_i}{1-\beta_i}\1(\tilde\tau_{SU}\geqslant \beta_{i+1})\right].$ Since $\tilde\tau_{SU}$ is a reverse stopping time w.r.t. filtration $\DG_t^T$ the value $\1(\tilde\tau_{SU}\geqslant \beta_{i+1})$ is $\DG_{\beta_{i+1}}-$measurable. Further, using the reverse martingale property of $\tilde M(t)$ we get
\begin{align*}
\E\left[\tilde M(\beta_i)\frac{\beta_i}{1-\beta_i}\1(\tilde\tau_{SU}\geqslant \beta_{i+1})\right]&=\frac{\beta_i}{1-\beta_i}\E\left[\1(\tilde\tau_{SU}\geqslant \beta_{i+1})\E\left[\tilde M(\beta_i)|\DG_{\beta_{i+1}}\right]\right]\\
&=\frac{\beta_i}{1-\beta_i}\E\left[\1(\tilde\tau_{SU}\geqslant \beta_{i+1})\tilde M(\beta_{i+1})\right].
\end{align*} 
Consequently, continuing the chain of equalities (\ref{rmenfr})-(\ref{rmenfr1}) we get
\begin{align}
\begin{aligned}
&\E\left[\tilde M(\tilde\tau_{SU})\frac{\tilde\tau_{SU}}{1-\tilde\tau_{SU}}\right]\\
&=\sum\limits_{i=1}^{n}\left(\frac{\beta_i}{1-\beta_i}\E\left[\tilde M(\beta_i)\1(\tilde\tau_{SU}\geqslant \beta_i)\right]-\frac{\beta_i}{1-\beta_i}\E\left[\tilde M(\beta_{i+1})\1(\tilde\tau_{SU}\geqslant \beta_{i+1})\right]\right)\\
\end{aligned}
\end{align}
\begin{align}\label{rmsum}
\begin{aligned}
&=\sum\limits_{i=1}^{n}\left(\frac{\beta_i}{1-\beta_i}\E\left[\tilde M(\beta_i)\1(\tilde\tau_{SU}\geqslant \beta_i)\right]-\frac{\beta_{i+1}}{1-\beta_{i+1}}\E\left[\tilde M(\beta_{i+1})\1(\tilde\tau_{SU}\geqslant \beta_{i+1})\right]\right)\\
&+\sum\limits_{i=1}^{n}\left(\left(\frac{\beta_{i+1}}{1-\beta_{i+1}}-\frac{\beta_i}{1-\beta_i}\right)\E\left[\tilde M(\beta_{i+1})\1(\tilde\tau_{SU}\geqslant \beta_{i+1})\right]\right)\\
&=\sum\limits_{i=1}^{n}\left(\left(\frac{\beta_{i+1}}{1-\beta_{i+1}}-\frac{\beta_i}{1-\beta_i}\right)\E\left[\tilde M(\beta_{i+1})\1(\tilde\tau_{SU}\geqslant \beta_{i+1})\right]\right).
\end{aligned}
\end{align}
The first sum in (\ref{rmsum}) vanishes due to telescoping summation and the fact that $\E\left[\tilde M(\beta_1)\right]=0$. Note also that $\E\left[\tilde M(\beta_{n+1})\1(\tilde\tau_{SU}\geqslant\beta_{n+1})\right]=0$ because $\1(\tilde\tau_{SU}\geqslant\beta_{n+1})=0$. Now, we have to show
\begin{align}\label{Y.6}
\sum\limits_{i=1}^{n-1}B_i\geqslant \alpha(n+1)P(V(\tilde\tau_{SU})=n_0) \text{ \ for \ }n_0\leqslant f(n).
\end{align}
where
\begin{align*}
B_i=\left(\frac{\beta_{i+1}}{1-\beta_{i+1}}-\frac{\beta_i}{1-\beta_i}\right)\E\left[\tilde M(\beta_{i+1})\1(\tilde\tau_{SU}\geqslant \beta_{i+1})\right], \ i\leqslant n-1.
\end{align*}
First, note that all $B_i$ are non-negative for all $i\leqslant n-1,$ due to Lemma \textcolor{magenta}{16}. Moreover, note that$$\tilde\tau_{SU}=\beta_{n}\iff \ V(\beta_{n})=n_0$$ under the Dirac distribution of $(p_i)_{i\in I_1}$ and consider the last summand 
\begin{align*}
&B_{n-1}=\left(\frac{\beta_{n}}{1-\beta_{n}}-\frac{\beta_{n-1}}{1-\beta_{n-1}}\right)\E\left[\tilde M(\beta_{n})\1(\tilde\tau_{SU}\geqslant \beta_{n})\right]\\
&=\left(\frac{\beta_{n}}{1-\beta_{n}}-\frac{\beta_{n-1}}{1-\beta_{n-1}}\right)\E\left[\left(\frac{n_0}{\beta_n}-n_0\right)\1(\tilde\tau_{SU}= \beta_{n})\right]\\
&=\left(1-\frac{\beta_{n-1}}{1-\beta_{n-1}}\frac{1-\beta_{n}}{\beta_{n}}\right)n_0P\left(V(\tilde\tau_{SU})=n_0\right)\\
&=\frac{n+3}{2(n+1)}n_0P\left(V(\tilde\tau_{SU})=n_0\right)\geqslant \alpha(n+1)P(V(\tilde\tau_{SU})=n_0)
\end{align*}
if $n_0\geqslant f(n)=\frac{2\alpha(n+1)^2}{n+3}.$ This completes the proof of part (a).\newline
(b) The second part follows immediately from (a) and from the formula for FDR of SU procedure based on the set of critical values $(\beta_i)_{i\leqslant n}$ under the reverse martingale model:
$$\text{FDR}=\frac{\alpha n_0}{n+1}+\frac{1-\alpha}{n+1}\E\left[V\right ].$$
\end{proof}
\section*{Appendix. Examples of martingale models.}
The family of (super-)martingales is a rich class of models which is briefly reviewed below. In this section we present a couple of examples. Further examples can be found in Heesen and Janssen \cite{h_j_15.1} and Benditkis \cite{me}. For convenience let us describe the model in this section by distributions $P$ on $[0,1]^{n},$ where the coordinates $(p_1,...,p_n)\in[0,1]^{n}$ represent p-values. We restrict ourselves to martingale models
\begin{align}\label{z.1}
\left([0,1]^{n},P,(\DF_t^T)_{t\in T}\right).
\end{align}
 Let $\mathcal{M}_{I_0,I_1}^T$ denote the set of martingale models $P$ on $[0,1]^n$ for fixed portion $I_0\neq\emptyset, \ I_1$ of $\{1,...,n\}$ and $\{0\}\subset T\subset [0,\eta]$ for some $0<\eta<1.$\newline
Obviously, there is a one to one correspondence between martingales and reverse martingales via the transformation 
\begin{align}\label{z.2}
\tilde p_i=1-p_i, \ \tilde T=\{1-t:t\in T\}, \DG_t^{\tilde T}:=\sigma(\1(p_i\geqslant s), \ s\geqslant t, \ s,t\in \tilde T)
\end{align}
of (\ref{z.1}). Note also that $(p_i)_{i\in I_0}$ follow special copula models if each $p_i, \ i\in I_0,$ is uniformly distributed on $(0,1).$\newline
To warm up consider first some useful elementary examples which will be combined below.
\begin{example}\label{Z.1}
\begin{itemize}
\item[(a)](Marshall and Olkin type dependence (see Marshall and Olkin \cite{Marshall})) Let $X_1,...,X_n$ be i.i.d. continuous distributed real random variables and $Y$ be a continuously distributed real random variable independent of $X_1,...,X_n.$  Consider $Z_i := \min(X_i,Y )$ and $\tilde Z_i := \max(X_i,Y )$ for $1 \leqslant i \leqslant n$. The transformed true p-values $p_i := H(Zi)$ and $\tilde p_i := \tilde H(Zi), \ i = 1,...,n$ fulfil the martingale property, reverse martingale property, respectively. Thereby, $H$ and $\tilde H$ are distribution functions of $Z_i$ and $\tilde{Z}_i, \ i\leqslant n.$
\item[(b)](Block models) Suppose that the index set 
\begin{align*}
\{1,...,n\}=\sum\limits_{j=1}^{k}(I_{0,j}+I_{1,j})
\end{align*}
splits in $k$ disjoint portions of $I_{0,j}$ the true and $I_{1,j}$ false null. Let $U_1,...,U_k$ denote i.i.d uniformly distributed random variables on $(0,1).$ Suppose that \newline $\left((U_1,(p_i)_{i\in I_{1,1}}),(U_2,(p_i)_{i\in I_{1,2}}),...,(U_k,(p_i)_{i\in I_{1,k}})\right))$ are independent martingale models of dimension $|I_{1,j}|+1$ for $j\leqslant k$. The $U$'s can be duplicated by the definition
\begin{align*}
p_i=U_j \ \text{if} \ i\in I_{0,j}, \  I_0=\sum\limits_{j=1}^{k} I_{0,j},
\end{align*} 
and we arrive at a martingale model where $\left(p_i\right)_{i\in\sum\limits_{j=1}^{k}I_{1,j}}$ are already defined.
\end{itemize}
\end{example}
Let us summarize further results.
\begin{example}
\begin{itemize}
\item[(a)] $M_{I_0,I_1}^T$ is closed under convex combinations.
\item[(b)] New martingale models can be produced by stopped martingales via stopping times and the optional switching device, see Heesen and Janssen \cite{h_j_15.1}, p.685.
\item[(c)](Martingales and financial models) Let $T\subset [0,\eta], \ \eta<1$ be a set with $0\in T$. Introduce the price process
\begin{align*}
X_t:[0,1)^{n}\rightarrow [0,\infty), \ X_t(p_1,...,p_n)=\sum\limits_{i\in I_0}\left(\frac{\1(p_i\leqslant t)-t}{1-t}+K\right), \ t\in T
\end{align*}
on $T$ for some constant $K\geqslant \max(\frac{s}{1-s}), \ s\in T.$ Then the process $t\mapsto X_t$ can be viewed as a discounted price process for time points $t\in T.$ The existence of martingale measures for $\left(X_t,\DF_t\right)_{t\in T}$ on the domain $[0,1]^{n}$ is well studied  in mathematical finances, see Shiryaev (1999). When the parameter set $T$ is finite it turns out that the space of probability measures on $[0,1]^{n}$, making that process to be a martingale, is of infinite dimension.
\item[(d)](Super-martingales)
It is well known that the process
\begin{align}\label{z.4}
\sum\limits_{i\in I_0}\left(\frac{\1(p_i\leqslant t)-t}{1-t}\right)=M_t+A_t, \ t\in T
\end{align}
admits a Doob-Meyer decomposition given by a $(\DF_t)_{t\in T}$ martingale $t\mapsto M_t$ and a compensator $t\mapsto A_t$ which is predictable with $A_t=0.$ Note that (\ref{z.4}) is a supermartingale if $t\mapsto A_t$ is non-increasing.
\end{itemize}
\end{example}

\section*{Acknowledgements}
The authors are grateful to Helmut Finner, Veronika Gontscharuk, Marsel Scheer and Philipp Heesen for many stimulating discussions.


\end{document}